\documentclass{article}

\usepackage[T1]{fontenc}
\usepackage[latin9]{inputenc}
\usepackage{units}
\usepackage{amsthm}
\usepackage{breqn}
\usepackage{graphicx}
\usepackage{amssymb}

\newtheorem{theorem}{Theorem}[section]
\newtheorem{lemma}[theorem]{Lemma}

\newtheorem{remark}[theorem]{Remark}

\newtheorem{conjecture}[theorem]{Conjecture}

\begin{document}

\title{Bounds on the Information Divergence\\ for Hypergeometric Distributions}

\author{Peter Harremo{\" e}s and Franti{\v s}ek Mat{\' u}{\v s}}



\maketitle
\begin{abstract}
The hypergeometric distributions have many important applications,
but they have not had sufficient attention in information theory.
Hypergeometric distributions can be approximated by binomial distributions
or Poisson distributions. In this paper we present upper and lower
bounds on information divergence. These bounds are important for statistical
testing and a better understanding of the notion of exchange-ability. 
\end{abstract}



\section{Introduction }

If a sample of size $n$ is taken from a population of size $N$ that
consist of $K$ white balls and $N-K$ black balls then the number
of white balls in the sample has a hypergeometric distribution that
we will denote $hyp\left(N,K,n\right)$. This type of sampling without
replacement is the standard example of an exchangeable sequence. The
point probabilities are 
\[
\Pr\left(X=x\right)=\frac{\binom{K}{x}\binom{N-K}{n-x}}{\binom{N}{n}}\,.
\]
The hypergeometric distribution also appears as a count in a contingency
table under the hypothesis of independence. Therefore the hypergeometric
distribution plays an important role for testing independence and
it was shown in \cite{Harremoes2014} that the mutual information
statistic for these distributions have distributions that are closer
to $\chi^{2}$-distributions than the distribution of the classical
$\chi^{2}$-statistics.

\begin{figure}
\centering{}\includegraphics[scale=0.5]{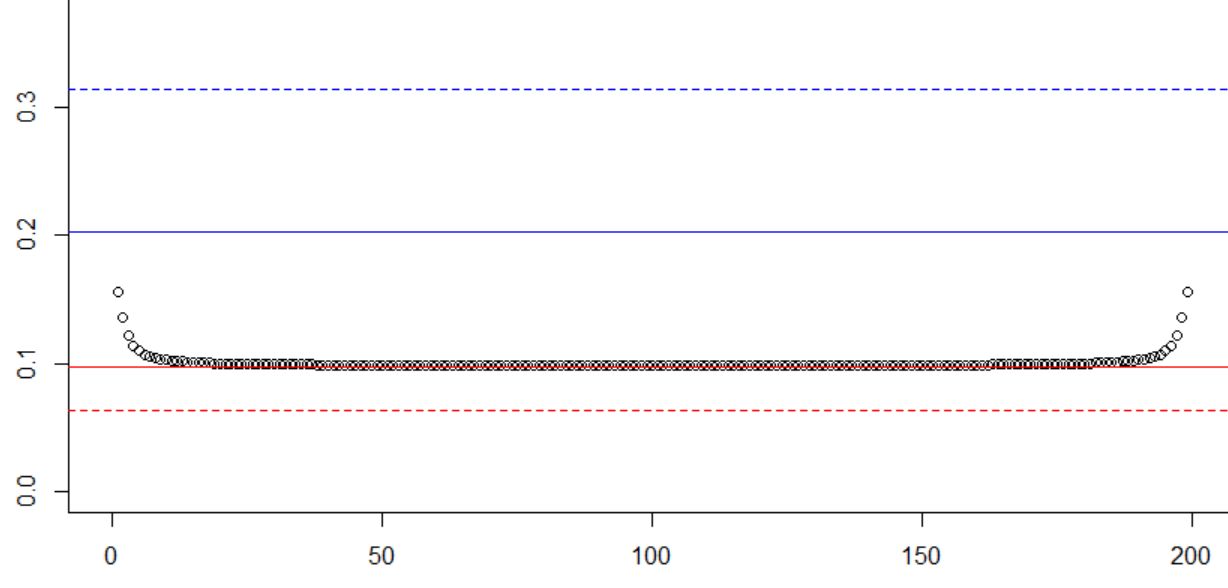}
\caption{\label{fig:1}Plot of the divergence of the hypergeometric distribution
$hyp(200,K,101)$ from the binomial distribution $bin(101;\nicefrac{K}{200})$
as a function of the number of white balls $K.$ The straight dashed
lines are the upper bound and the lower bound proved
by Stam. The solid lines are the upper bound and the lower bound proved in this paper. The
plot illustrates that a function that does not depend on $K$ can
give a very precise lower bound for most values of $K$, but a good
upper bound should depend on $K.$}
\end{figure}

Hypergeometric distributions do not form an exponential family. For
this and other reasons one often try to approximate the hypergeometric
distribution by a binomial distribution or a Poisson distribution.
This technique was also used in \cite{Harremoes2014}. In the literature
one can find many bounds on the total variation between hypergeometric
distributions and binomial distributions or Poisson distributions
\cite{Barbour1992}, but until recently there was only one paper by
Stam \cite{Stam1978} where the information divergence of a hypergeometric
distribution from a binomial distribution is bounded. As we will demonstrate
in this paper the bounds by Stam can be improved significantly. Precise
bounds are in particular important for testing because the error probability
is asymptotically determined by information divergence via Sanov's
Theorem \cite{Cover1991,Csiszar2004}. The bounds in this paper supplement
the bounds by Mat{\' u}{\v s} \cite{Matus2017}. 

We are also interested in the multivariate hypergeometric distribution
that can be approximated by a multinomial distribution. Instead of
two colors we now consider the situation where there are $C$ colors.
Again, we let $n$ denote the sample size and we let $N$ denote the
population size. Now we may consider sampling with or without replacement.
Without replacement we get a multivariate hypergeometric distribution
and with replacement we get a multinomial distribution. Stam proved
the following upper bound on the divergence
\begin{equation}
D\left(\left.hyp\right\Vert mult\right)\leq\left(C-1\right)\frac{n\left(n-1\right)}{2\left(N-1\right)\left(N-n+1\right)}\label{eq:Stam}
\end{equation}
This bound is relatively simple and it does not depend on the number
of balls of each color. Stam also derived the following lower bound,
\begin{equation}
D\left(\left.hyp\right\Vert mult\right)\geq\left(C-1\right)\frac{n\left(n-1\right)}{2\left(N-1\right)^{2}}\cdot\left(\frac{1}{2}+\frac{1}{6}\cdot\frac{Q}{C-1}\cdot\frac{N-2n+2}{\left(N-n+1\right)\left(N-2\right)}\right)\label{eq:Stamnedre}
\end{equation}
where $Q$ is a positive constant depending on the number of balls
of each color. If $\nicefrac{n}{N}$ is not close to zero there is
a significant gap between his lower bound and his upper bound. Therefore
it is unclear whether his lower bound or his upper bound gives the
best approximation of information divergence. In this paper we will
derive the correct asymptotic expression for information divergence
(Theorem \ref{thm:Asymptotisk}). We will derive relatively simple
lower bounds. We have not achieved simple expressions for
upper bounds that are asymptotically tight, but we prove that our
simple lower bounds are asymptotically tight. The problem with complicated
upper bounds seems to be unavoidable if they should be asymptotically
tight. At least the same pattern showed up for approximation of binomial
distributions by Poisson distributions \cite{Harremoes2004}. 

\begin{figure}
\begin{centering}
\includegraphics[scale=0.4]{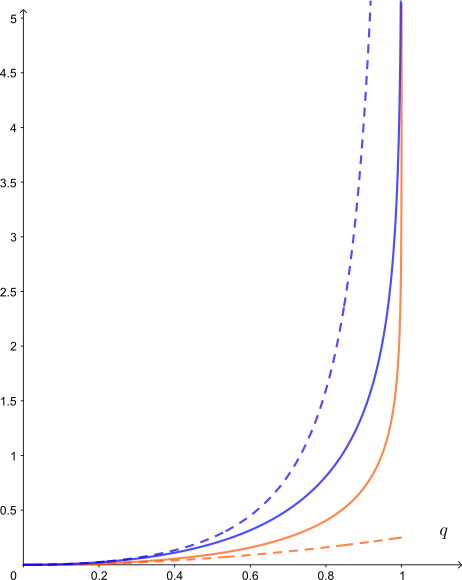}
\par\end{centering}
\caption{The figure illustrates the lower bounds and the upper bounds. The bounds given in this paper are solid while the the bounds
given by Stam are dashed. The bounds are calculated for large values
of $n$ and $N$ and the bounds are given as function of $q=\nicefrac{n}{N}.$ The
bounds of Stam are good for smal values of $q$, but for values of
$q$ close to 1 the bounds of Stam have been improved significantly.}
\end{figure}

Our upper bound on information divergence also leads to upper bounds
on total variation. Such bounds are important for the study of finite
exchange-ability compared with infinite exchange-ability \cite{Diaconis1987},
but this application will not be discussed in the present paper. 

\section{Lower bound for a Poisson approximation }

The hypergeometric distribution $hyp(N,K,n)$ has mean value $\frac{n\cdot K}{N}$
and variance 
\[
\frac{nK\left(N-n\right)\left(N-K\right)}{N^{2}\left(N-1\right)}.
\]
If $N$ is large compared with $n$ and with $K$, we may approximate
the hypergeometric distribution by a Poisson distribution with mean
$\frac{n\cdot K}{N}.$
\begin{theorem}
The divergence of the hypergeometric distribution $hyp(N,K,n)$ from
the Poisson distribution $Po\left(\lambda\right)$ with $\lambda=\frac{n\cdot K}{N}$
satisfies the following lower bound
\[
D\left(\left.hyp\left(N,K,n\right)\right\Vert Po\left(\lambda\right)\right)\geq\frac{1}{2}\left(\frac{K+n-\lambda-1}{N-1}\right)^{2}.
\]
\end{theorem}

\begin{proof}
If $n=K=N$ then $\lambda=N$ and the inequality states that
\[
D\left(\left.hyp\left(N,N,N\right)\right\Vert Po\left(N\right)\right)\geq\frac{1}{2}.
\]
In this case the hypergeometric distribution attains the value $N$
with probability 1 and the divergence has value
\begin{dmath}
-\ln\left(\frac{N^{N}}{N!}\exp\left(-N\right)\right)  \geq-\ln\left(\frac{N^{N}}{\tau^{\nicefrac{1}{2}}N^{N+\nicefrac{1}{2}}\exp\left(-N\right)}\exp\left(-N\right)\right).\\
  =\frac{1}{2}\ln\left(\tau\right)+\frac{1}{2}\ln\left(N\right)\\
  \geq\frac{1}{2}\ln\left(\tau\right).
\end{dmath}
Here we have used the lower bound in the Stirling approximation and
used $\tau$ as short for $2\pi.$ In this special case the result
follows because $\tau>\mathrm{e}.$

Therefore we may assume that $n<N$ or $K<N.$ Harremo{\" e}s, Johnson
and Kontoyannis \cite{Harremoes2016b} have proved that if a random
variable $X$ satisfies $E\left[X\right]=\lambda$ and $Var$$\left(X\right)\leq\lambda$
then 
\[
D\left(\left.X\right\Vert Po\left(\lambda\right)\right)\geq\frac{1}{2}\left(1-\frac{Var\left(X\right)}{\lambda}\right)^{2}.
\]
The variance of the hypergeometric distribution satisfies 
\begin{dmath}
\frac{nK\left(N-n\right)\left(N-K\right)}{N^{2}\left(N-1\right)}  =\lambda\frac{\left(N-n\right)\left(N-K\right)}{N\left(N-1\right)}\\
  \leq\lambda.
\end{dmath}
Now we get 
\begin{dmath}
D\left(\left.hyp\left(N,K,n\right)\right\Vert Po\left(\lambda\right)\right)  \geq\frac{1}{2}\left(1-\frac{\left(N-n\right)\left(N-K\right)}{N\left(N-1\right)}\right)^{2}
  =\frac{1}{2}\left(\frac{K+n-\lambda-1}{N-1}\right)^{2}.
\end{dmath}
\end{proof}
The lower bound can be rewritten as 
\[
D\left(\left.hyp\left(N,K,n\right)\right\Vert Po\left(\lambda\right)\right)\geq\frac{1}{2}\left(\frac{\frac{\lambda}{n}+\frac{\lambda}{K}-\frac{\lambda+1}{N}}{1-\frac{1}{N}}\right)^{2}.
\]
For a sequence of approximations with a fixed value of $\lambda,$
the lower bound will tend to zero if and only if both $n$ and $K$
tend to infinity. If only one of the parameters $n$ and $K$ tends
to infinity and the other is bounded or perhaps even constant, then
one would approximate the hypergeometric distribution by a binomial
distribution instead. 

\section{Lower bound for a binomial approximation }

One may compare sampling without replacement by sampling with replacement.
For parameters $N,K$ and $n$ it means that one may compare the hypergeometric
distribution $hyp(N,K,n)$ with the binomial distribution $bin(n,p)$
with $p=\nicefrac{K}{N}$. One can use the same technique as developed
in \cite{Harremoes2016b} to obtain a lower bound on information divergence.
This technique uses orthogonal polynomials. The \emph{Kravchuk polynomials}
are orthogonal polynomials with respect to the binomial distribution
$bin\left(n,p\right)$ and are given by
\[
\mathcal{K}_{k}\left(x;n\right)=\sum_{j=0}^{k}\left(-1\right)^{j}\left(\frac{p}{1-p}\right)^{k-j}\binom{x}{j}\binom{n-x}{k-j}\,.
\]

\begin{remark}
Often the parameter $q=\frac{1}{1-p}$ is used to parametrize the
Kravchuk polynomials, but we will not use this notion.
\end{remark}

The Kravchuk polynomials satisfy 
\begin{equation}
\sum_{i=0}^{n}\binom{n}{i}p^{i}\left(1-p\right)^{n-i}\mathcal{K}_{r}\left(x;n\right)\mathcal{K}_{s}\left(x;n\right)=\left(\frac{p}{1-p}\right)^{r}\binom{n}{r}\delta_{r,s}\,.\label{eq:ortogonal}
\end{equation}
The first three Kravchuk polynomials are 
\begin{dgroup*}\begin{dmath*}\mathcal{K}_{0}\left(x;n\right)  =1\,,\end{dmath*}
\begin{dmath*}\mathcal{K}_{1}\left(x;n\right)  =\frac{np-x}{1-p}\,,\end{dmath*}
\begin{dmath*}\mathcal{K}_{2}\left(x;n\right)  =\frac{\left(2p-1\right)\left(x-np\right)+\left(x-np\right)^{2}-np\left(1-p\right)}{2\left(1-p\right)^{2}}.
\end{dmath*}
\end{dgroup*}
For a random variable $X$ with mean value $np$ one has 
\[
E\left[\mathcal{K}_{2}\left(X;n\right)\right]=\frac{Var\left(X\right)-np\left(1-p\right)}{2\left(1-p\right)^{2}}
\]
so the second Kravchuk moment measures how much a random variable
with mean $np$ deviates from having variance $np\left(1-p\right)$.
We need to calculate moments of the Kravchuk polynomials with respect
to a binomial distribution. Let $X$ denote a binomial random variable
with distribution $bin(n,p).$ : The first moment is easy 
\[
E\left[\mathcal{K}_{1}\left(X;n\right)\right]=0.
\]
The second moment can be calculated from Equation (\ref{eq:ortogonal})
and is 
\[
E\left[\mathcal{K}_{2}\left(X;n\right)\right]=\frac{p^{2}}{\left(1-p\right)^{2}}\binom{n}{2}\,.
\]
The \emph{normalized Kravchuk polynomial} of order 2 is 
\begin{dmath*}
\tilde{\mathcal{K}}_{2}\left(x;n\right)  =\frac{\frac{\left(2p-1\right)\left(x-np\right)+\left(x-np\right)^{2}-np\left(1-p\right)}{2\left(1-p\right)^{2}}}{\left(\frac{p^{2}}{\left(1-p\right)^{2}}\binom{n}{2}\right)^{\nicefrac{1}{2}}}
  =\frac{\frac{2p-1}{np\left(1-p\right)}\left(x-np\right)+\frac{\left(x-np\right)^{2}}{np\left(1-p\right)}-1}{\left(2\frac{n-1}{n}\right)^{\nicefrac{1}{2}}}.
\end{dmath*}
The minimum of the normalized Kratchuk polynomial is
\begin{equation}
-\frac{\frac{\left(\frac{1}{2}-p\right)^{2}}{np\left(1-p\right)}+1}{\left(2\frac{n-1}{n}\right)^{\nicefrac{1}{2}}}.\label{eq:min}
\end{equation}

If $X$ is a hypergeometric random variable then
\begin{dmath*}
E\left[\tilde{\mathcal{K}}_{2}\left(X;n\right)\right]  =\frac{\frac{n\frac{K}{N}\frac{N-K}{N}\cdot\frac{N-n}{N-1}}{np\left(1-p\right)}-1}{\left(2\frac{n-1}{n}\right)^{\nicefrac{1}{2}}}
  =\frac{\frac{N-n}{N-1}-1}{\left(2\frac{n-1}{n}\right)^{\nicefrac{1}{2}}}
  =\textrm{-}\frac{\left(n\left(n-1\right)\right)^{\nicefrac{1}{2}}}{2^{\nicefrac{1}{2}}\left(N-1\right)}.
\end{dmath*}
We note that $E\left[\tilde{\mathcal{K}}_{2}\left(X;n\right)\right]\geq\textrm{-}2^{\textrm{-}\nicefrac{1}{2}}$
as long as $n<N.$ 

For any (positive) discrete measures $P$ and $Q$ information divergence
is defined as
\[
D\left(P\left\Vert Q\right.\right)=\sum_{i}p_{i}\ln\frac{p_{i}}{q_{i}}-p_{i}+q_{i}.
\]
For a fixed measure $Q$ the measure that minimizes the $D\left(P\left\Vert Q\right.\right)$
under a linear constraint $\sum x_{i}\cdot p_{i}=m$ is the measure
$Q_{\beta}$ with point masses $\exp\left(\beta\cdot x_{i}\right)\cdot q_{i}$
for which 
\[
\sum_{i}x_{i}\exp\left(\beta\cdot x_{i}\right)\cdot q_{i}=m.
\]
We introduce the moment generating function $M\left(\beta\right)=\sum_{i}\exp\left(\beta\cdot x_{i}\right)\cdot q_{i}$
and observe that 
\[
M^{\left(k\right)}\left(\beta\right)=\sum_{i}x_{i}^{k}\exp\left(\beta\cdot x_{i}\right)\cdot q_{i}.
\]

\begin{theorem}
\label{thm:epsilon}For any binomial distribution $bin(n;p)$ there
exists an $\epsilon>0$ such that for any measure $P$ and with $E_{P}\left[\tilde{\mathcal{K}}_{2}\left(X;n\right)\right]\in\left]\textrm{-}\epsilon,0\right]$
one has 
\[
D\left(P\left\Vert bin\left(n,p\right)\right.\right)\geq\frac{\left(E_{P}\left[\tilde{\mathcal{K}}_{2}\left(X;n\right)\right]\right)^{2}}{2}
\]
where 
\[
E_{P}\left[\tilde{\mathcal{K}}_{2}\left(X;n\right)\right]=\sum_{x=0}^{n}\tilde{\mathcal{K}}_{2}\left(x;n\right)P\left(x\right).
\]
\end{theorem}

\begin{proof}
Let $M$ denote the moment generating function 
\[
M\left(\beta\right)=\sum_{x=0}^{n}\exp\left(\beta\cdot\tilde{\mathcal{K}}_{2}\left(x;n\right)\right)bin\left(n,p,x\right)
\]
and let $Q_{\beta}$ denote the measure with 
\[
Q_{\beta}\left(x\right)=\exp\left(\beta\cdot\tilde{\mathcal{K}}_{2}\left(x;n\right)\right)bin\left(n,p,x\right)
\]
and $M'\left(\beta\right)=\mu.$ We have 
\begin{dmath*}
D\left(P\left\Vert bin\left(n,p\right)\right.\right)  \geq D\left(Q_{\beta}\left\Vert bin\left(n,p\right)\right.\right)=\beta\cdot\mu-\left(M\left(\beta\right)-1\right)
\end{dmath*}
so we want to prove that 
\[
\beta\cdot\mu-\left(M\left(\beta\right)-1\right)\geq\frac{1}{2}\mu^{2}.
\]
This inequality holds for $\mu=0$ so we differentiate with respect
to $\beta$ and see that the it is sufficient to prove that 
\begin{dmath*}
\mu+\beta\cdot\frac{\mathrm{d}\mu}{\mathrm{d}\beta}-M'\left(\beta\right)  \leq\mu\cdot\frac{\mathrm{d}\mu}{\mathrm{d}\beta},
\beta\cdot\frac{\mathrm{d}\mu}{\mathrm{d}\beta}  \leq\mu\cdot\frac{\mathrm{d}\mu}{\mathrm{d}\beta},
\beta  \leq\mu.
\end{dmath*}
We differentiate once more with respect to $\beta$  and see that it is
sufficient to prove that
\[
1\geq\frac{\mathrm{d}\mu}{\mathrm{d}\beta}.
\]
Now $\frac{\mathrm{d}\mu}{\mathrm{d}\beta}=M''\left(\beta\right).$
Since $M''\left(0\right)=1$ it is sufficient to prove that 
\[
M^{\left(3\right)}\left(0\right)=E\left[\left(\tilde{\mathcal{K}}_{2}\left(X;n\right)\right)^{3}\right]>0
\]
if $X$ is binomial $bin\left(n,p\right).$ 

Up to a positive factor the third moment of the Kravchuk polynomial
is given by
\begin{dmath*}
\left(2\left(1-p\right)^{2}\mathcal{K}_{2}\left(n;x\right)\right)^{3}=  \textrm{-}n^{3}p^{3}\left(1-p\right)^{3}+3n^{2}p^{2}\left(1-p\right)^{2}\left(2p-1\right)\left(x-np\right)\\
  +np\left(1-p\right)\left(4np\left(1-p\right)-3\left(2p-1\right)^{2}\right)\left(x-np\right)^{2}\\
  +\left(2p-1\right)^{2}\left(\left(2p-1\right)+2np\right)\left(x-np\right)^{3}\\
  +3\left(\left(2p-1\right)^{2}-np\left(1-p\right)\right)\left(x-np\right)^{4}\\
  +3\left(2p-1\right)\left(x-np\right)^{5}+\left(x-np\right)^{6}.
\end{dmath*}
 Using the values of the first six central moments of the binomial
distribution we get 
\begin{dmath*}
E\left[\left(2\left(1-p\right)^{2}\mathcal{K}_{2}\left(n;x\right)\right)^{3}\right]
=np^{2}\left(1-p\right)^{2}\left(8n-2+p\left(1-p\right)\left(89n^{2}-293n+174\right)\right).
\end{dmath*}
If $n>2$ we have $89n^{2}-293n+174>0$ so the whole expression becomes
positive. For $n=2$ the last factor equals $14-56p\left(1-p\right),$which
is positive except for $p=\nicefrac{1}{2}$ where it equals zero.
\end{proof}
For the hypergeometric distributions one gets the lower bound 
\begin{dmath}
D\left(hyp\left(N,K,n\right)\left\Vert bin\left(n,\frac{K}{N}\right)\right.\right)  \geq\frac{\left(\textrm{-}\frac{\left(n\left(n-1\right)\right)^{\nicefrac{1}{2}}}{2^{\nicefrac{1}{2}}\left(N-1\right)}\right)^{2}}{2}\nonumber 
  =\frac{n\left(n-1\right)}{4\left(N-1\right)^{2}}\,.\label{eq:1dimHybLower}
\end{dmath}
According to Theorem \ref{thm:epsilon} this inequality holds if $N$ is sufficiently
large, but later (Theorem \ref{thm:kvadratiskLower}) we shall see
that this lower bound (\ref{eq:1dimHybLower}) holds for hypergeometric distribution for any value of $N.$
\begin{theorem}
\label{thm:Lower1dim}Assume that the parameters of the binomial distribution
are such that $np$ is an integer. Let $X$ denote a random
variable such that 
\[
\textrm{-}2^{\textrm{-}\nicefrac{1}{2}}\leq E\left[\mathcal{K}_{2}\left(n,X\right)\right]\leq0.
\]
Then 
\begin{equation}
D\left(P\left\Vert bin\left(n,p\right)\right.\right)\geq\frac{\left(E\left[\tilde{\mathcal{K}}_{2}\left(X;n\right)\right]\right)^{2}}{2}.\label{eq:3}
\end{equation}
where $P$ denotes the distribution of $X.$
\end{theorem}

\begin{proof}
As in the proof of Theorem \ref{thm:epsilon} it is sufficient to
prove that 
\[
M''\left(\beta\right)\leq1.
\]
The function
\[
\beta\to M''\left(\beta\right)=\sum_{x=0}^{n}\tilde{\mathcal{K}}_{2}\left(n;x\right)^{2}\exp\left(\beta\tilde{\mathcal{K}}_{2}\left(n;x\right)\right)bin\left(n,p,x\right)
\]
is convex in $\beta$, so if we prove the inequality $M''\left(\beta\right)\leq1$
for $\beta=0$ and for $\beta=\beta_{0}<0$ then the inequality holds
for any $\beta\in\left[\beta_{0},0\right].$ Let $\beta_{0}$ denote
the constant $\textrm{-}\nicefrac{2}{\mathrm{e}}\,.$ We observe that
$\beta_{0}$ is slightly less than $\textrm{-}2^{\textrm{-}\nicefrac{1}{2}}.$ 

Consider the function $f\left(x\right)=x^{2}\exp\left(\beta_{0}x\right)$
with 
\[
f'\left(x\right)=\left(2+\beta_{0}x\right)x\exp\left(\beta_{0}x\right)\,.
\]
The function $f$ is decreasing for $x\leq0$, it has minimum 0 for
$x=0$, it is increasing for $0\leq x\leq-\nicefrac{2}{\beta_{0}=\mathrm{e}}$,
it has local maximum 1 for $x=\mathrm{e},$ and it is decreasing for
$x\geq\mathrm{e}.$ We have $f\left(\beta_{0}\right)=\frac{4}{\exp2}\exp\left(\frac{4}{\exp2}\right)<1.$
Hence $f\left(x\right)\leq1$ for $x\geq\beta_{0}.$

The graph of $x\to\tilde{\mathcal{K}}_{2}\left(n;x\right)$ is a parabola.
We note that 
\[
\frac{\mathrm{d}}{\mathrm{d}x}\mathcal{K}_{2}\left(x;n\right)=\frac{p-\frac{1}{2}+x-np}{\left(1-p\right)^{2}}
\]
so as a function with real domain there is a stationary point at 
\[
x=\left(n-1\right)p+\frac{1}{2}\,.
\]
Since a binomial distribution can only take integer values the minimum
is attained for the integer in the interval $\left[\left(n-1\right)p,\left(n-1\right)p+1\right],$
but the integer $np$ is the only integer in this interval. Therefore
for $x\in\mathbb{Z}$ the minimum of $\tilde{\mathcal{K}}_{2}\left(n;x\right)$
is 
\[
\tilde{\mathcal{K}}_{2}\left(n;np\right)=\frac{\textrm{-}1}{\left(2\frac{n-1}{n}\right)^{\nicefrac{1}{2}}}
\]
so the inequality holds as long as
\[
\beta_{0}\leq\frac{\textrm{-}1}{\left(2\frac{n-1}{n}\right)^{\nicefrac{1}{2}}}\,.
\]
We isolate $n$ in this inequality and get
\[
n\geq\frac{1}{1-\frac{1}{2\beta_{0}^{2}}}=\frac{8}{8-\mathrm{e}^{2}}=13.0945>13.
\]
If $n\leq13$ and $np$ is an integer then there are only 91 cases
and in each of these cases we can numerically check Inequality (\ref{eq:3}). 
\end{proof}
\begin{conjecture}
We conjecture that Theorem \ref{thm:Lower1dim} holds without the
conditions that $np$ is an integer. 
\end{conjecture}

\section{Improved bounds on information divergence\\ for multivariate hypergeometric
distributions }

We consider the situation where there are $N$ balls of $C$ different
colors. Let $k_{c}$ denote the number of balls of color $c\in\left\{ 1,2,\dots,C\right\} $
and let $p_{c}=\nicefrac{k_{c}}{N}$. Let $U_{n}$ denote the number
of balls in different colors drawn without replacement in a sample
of size $n$ and let $V_{n}$ denote the number of balls for different
colors drawn with replacement. Then $U_{n}$ has a multivariate hypergeometric
distribution and $V_{n}$ has a multinomial distribution. We are interested
in bounds on information divergence that we, with a little abuse of
notation, will denote $D\left(U_{n}\left\Vert V_{n}\right.\right)$.
We consider $U_{n}$ as a function of $X^{n}$ where $X^{n}=\left(X_{1},X_{2},\dots,X_{n}\right)$
denotes a sequence colors in the sample drawn without replacement.
Similarly we consider $V_{n}$ as a function of $Y^{n}$ where $Y^{n}=\left(Y_{1},Y_{2},\dots,Y_{n}\right)$
denotes a sequence of colors drawn with replacement. Let $I\left(\cdot,\cdot\mid\cdot\right)$
denote conditional mutual information. 
\begin{lemma}
\label{lem:mutualExp}We have 
\[
D\left(\left.U_{n}\right\Vert V_{n}\right)=\sum_{j=1}^{n-1}\left(n-j\right)I\left(X^{j},X_{j+1}\mid X^{j-1}\right)
\]
\end{lemma}

\begin{proof}
We have 
\begin{dmath*}
D\left(\left.U_{n}\right\Vert V_{n}\right)  =D\left(\left.X^{n}\right\Vert Y^{n}\right)\\
  =I\left(X_{1},X_{2},\dots,X_{n}\right)\\
  =\sum_{m=1}^{n-1}I\left(X^{m},X_{m+1}\right)\\
  =\sum_{m=1}^{n-1}\sum_{j=1}^{m}I\left(X^{j},X_{m+1} \hiderel\mid X^{j-1}\right).
\end{dmath*}
Using exchange-ability we get
\begin{dmath*}
D\left(\left.U_{n}\right\Vert V_{n}\right)  =\sum_{m=1}^{n-1}\sum_{j=1}^{m}I\left(X^{j},X_{j+1} \hiderel\mid X^{j-1}\right)\\
  =\sum_{j=1}^{n-1}\sum_{m=j}^{n-1}I\left(X^{j},X_{j+1} \hiderel\mid X^{j-1}\right)\\
  =\sum_{j=1}^{n-1}\left(n-j\right)I\left(X^{j},X_{j+1} \hiderel\mid X^{j-1}\right).
\end{dmath*}
\end{proof}
We introduce the $\chi^{2}$-divergence by
\begin{dmath*}
\chi^{2}\left(P,Q\right)  =\int\left(\frac{\mathrm{d}P}{\mathrm{d}Q}-1\right)^{2}\,\mathrm{d}Q\\
  =\int\frac{\mathrm{d}P}{\mathrm{d}Q}\,\mathrm{d}P-1\,.
\end{dmath*}
Stam used the inequality $D\left(P\Vert Q\right)\leq\chi^{2}\left(P,Q\right)$
to derive his upper bound (\ref{eq:Stam}). From Theorem \ref{thm:Lower1dim}
and inequality (\ref{eq:1dimHybLower}) we should aim at replacing
the denominator 
\[
2\left(N-1\right)\left(N-n+1\right)
\]
by an expression
closer to $4\left(N-1\right)^{2}.$

The bounds we have derived are based on the following sequence of
inequalities that are derived in Appendix A.
We use $\phi\left(x\right)=x\ln\left(x\right)-\left(x-1\right).$
\begin{dgroup*}
\begin{dmath}\phi\left(x\right)  \geq0\,,\label{eq:positiv}\end{dmath}
\begin{dmath}\phi\left(x\right)  \leq\left(x-1\right)^{2},\label{eq:chi}\end{dmath}
\begin{dmath}\phi\left(x\right)  \geq\frac{1}{2}\left(x-1\right)^{2}-\frac{1}{6}\left(x-1\right)^{3},\label{eq:nedre}\end{dmath}
\begin{dmath}\phi\left(x\right)  \leq\frac{1}{2}\left(x-1\right)^{2}-\frac{1}{6}\left(x-1\right)^{3}+\frac{1}{3}\left(x-1\right)^{4}.\label{eq:oevre}\end{dmath}
\end{dgroup*}

The first inequality (\ref{eq:positiv}) implies non-negativity of
information divergence and mutual information. The second inequality
(\ref{eq:chi}) can be used to derive to Stam's inequality (\ref{eq:Stam}),
but the higher order terms are needed to get the asymptotics right. 
\begin{lemma}\label{lem:MutualUpper}
The mutual information is bounded as
\[
\frac{C-1}{2\left(N-j\right)^{2}}\leq I\left(X^{j},X_{j+1}\mid X^{j-1}\right)\leq\frac{C-1}{\left(N-j\right)^{2}}.
\]
\end{lemma}

\begin{proof}
Without loss of generality we may assume that $j=1.$ In this case
the inequalities follow directly from the inequalities (\ref{eq:Stam}) 
and (\ref{eq:Stamnedre}) of Stam with $n=1$. For completeness we give the
whole proof in Appendix B. 
\end{proof}
Combining Lemma \ref{lem:mutualExp} with Lemma \ref{lem:MutualUpper}
leads to the inequalities
\[
\frac{C-1}{2}\sum_{j=1}^{n-1}\frac{n-j}{\left(N-j\right)^{2}}\leq D\left(U_{n}\left\Vert V_{n}\right.\right)\leq\left(C-1\right)\sum_{j=1}^{n-1}\frac{n-j}{\left(N-j\right)^{2}}.
\]
We see that the lower bound and the upper bound are are off by a factor
of 2. Figure \ref{fig:1} illustrates that this factor is unavoidable
if we want bounds that do not depend on the number of balls in each
color.

The following simple lower bound is stronger than the lower bound (\ref{eq:Stamnedre})
by Stam for $n>\nicefrac{N}{2}.$
\begin{theorem}
\label{thm:kvadratiskLower}For all $n$ the following lower bound
holds 
\begin{dmath}\label{eq:*}
D\left(U_{n}\left\Vert V_{n}\right.\right)  \geq\left(C-1\right)\frac{n\left(n-1\right)}{4\left(N-1\right)^{2}}.
\end{dmath}
\end{theorem}

\begin{proof}
We have 
\begin{dmath*}
D\left(U_{n}\left\Vert V_{n}\right.\right)  \geq\frac{C-1}{2}\sum_{j=1}^{n-1}\frac{n-j}{\left(N-j\right)^{2}}\\
  \geq\frac{C-1}{2\left(N-1\right)^{2}}\cdot\sum_{j=1}^{n-1}\left(n-j\right)\\
  =\frac{C-1}{2\left(N-1\right)^{2}}\cdot\frac{n\left(n-1\right)}{2}\\
  =\left(C-1\right)\frac{n\left(n-1\right)}{4\left(N-1\right)^{2}}\,.
\end{dmath*}
\end{proof}
An even stronger lower bound can be derived. Later we will prove that
the stronger lower bound is asymptotically optimal.
\begin{theorem}
For all $n\leq N$ the multivariate hypergeometric distribution satisfies
the following lower bound.
\begin{dmath*}
D\left(U_{n}\left\Vert V_{n}\right.\right)  \geq\left(C-1\right)\frac{\ln\left(\frac{N}{N-n+1}\right)-\frac{n-1}{N-1}}{2}.
\end{dmath*}
\end{theorem}

\begin{proof}
We use an integral to lower bound the sum.
\begin{dmath*}
D\left(\left.U_{n}\right\Vert V_{n}\right)  \geq\frac{C-1}{2}\sum_{j=1}^{n-1}\frac{n-j}{\left(N-j\right)^{2}}\\
  =\frac{C-1}{2}\left(\sum_{j=1}^{n-1}\frac{\left(N-j\right)-\left(N-n\right)}{\left(N-j\right)^{2}}\right)\\
  =\frac{C-1}{2}\left(\sum_{j=1}^{n-1}\frac{1}{N-j}-\sum_{j=1}^{n-1}\frac{N-n}{\left(N-j\right)^{2}}\right).
\end{dmath*}
Each of the sums can be bounded by an integral
\begin{dmath*}
\frac{C-1}{2}\left(\sum_{j=1}^{n-1}\frac{1}{N-j}-\left(N-n\right)\sum_{j=1}^{n-1}\frac{1}{\left(N-j\right)^{2}}\right)
\geq\frac{C-1}{2}\left(\int_{0}^{n-1}\frac{1}{N-x}\,\mathrm{d}x-\left(N-n\right)\int_{1}^{n}\frac{1}{\left(N-x\right)^{2}}\,\mathrm{d}x\right)
\geq\frac{C-1}{2}\left(\ln\left(\frac{N}{N-n+1}\right)-\left(N-n\right)\left(\frac{1}{N-n}-\frac{1}{N-1}\right)\right)
=\frac{C-1}{2}\left(\ln\left(\frac{N}{N-n+1}\right)-\left(\frac{N-1}{N-1}-\frac{N-n}{N-1}\right)\right)
=\left(C-1\right)\frac{\ln\left(\frac{N}{N-n+1}\right)-\frac{n-1}{N-1}}{2}\,.
\end{dmath*}
\end{proof}
\begin{theorem}
For $n\leq\nicefrac{N}{2}$ the multivariate hypergeometric distribution
satisfies the following lower bound.
\begin{dmath}
D\left(U_{n}\left\Vert V_{n}\right.\right)  \geq\left(C-1\right)\frac{r-1-\ln\left(r\right)}{2}.\label{eq:optLower}
\end{dmath}
where $r=1-\frac{n-1}{N-\nicefrac{1}{2}}.$
\end{theorem}

\begin{proof}
Since $n\leq\nicefrac{N}{2}$ the function $j\to\frac{n-j}{\left(N-j\right)^{2}}$
is concave and the sum can be lower bounded by an integral.
\begin{dmath*}
D\left(\left.U_{n}\right\Vert V_{n}\right)\geq\frac{C-1}{2}\sum_{j=1}^{n-1}\frac{n-j}{\left(N-j\right)^{2}}\geq\frac{C-1}{2}\left(\int_{\nicefrac{1}{2}}^{n-\nicefrac{1}{2}}\frac{n-x}{\left(N-x\right)^{2}}\,dx\right)
=\frac{C-1}{2}\left(\int_{\nicefrac{1}{2}}^{n-\nicefrac{1}{2}}\frac{1}{N-x}\,\mathrm{d}x-\left(N-n\right)\int_{\nicefrac{1}{2}}^{n-\nicefrac{1}{2}}\frac{1}{\left(N-x\right)^{2}}\,\mathrm{d}x\right)
=\frac{C-1}{2}\left(\ln\left(\frac{N-\nicefrac{1}{2}}{N-n+\nicefrac{1}{2}}\right)-\left(N-n\right)\left(\frac{1}{N-n+\nicefrac{1}{2}}-\frac{1}{N-\nicefrac{1}{2}}\right)\right)
=\frac{C-1}{2}\left(\ln\left(\frac{N-\nicefrac{1}{2}}{N-n+\nicefrac{1}{2}}\right)-\left(\frac{N-n}{N-n+\nicefrac{1}{2}}\right)\left(\frac{n-1}{N-\nicefrac{1}{2}}\right)\right)
\geq\frac{C-1}{2}\left(\ln\left(\frac{1}{1-\frac{n-1}{N-\nicefrac{1}{2}}}\right)-\frac{n-1}{N-\nicefrac{1}{2}}\right).
\end{dmath*}
\end{proof}
\begin{theorem}
The following inequality holds.
\[
D\left(\left.U_{n}\right\Vert V_{n}\right)\leq\left(C-1\right)\left(\ln\left(\frac{N-1}{N-n}\right)+\frac{1}{N-n+1}-\frac{n}{N}\right).
\]
\end{theorem}

\begin{proof}
We have 
\begin{dmath*}
D\left(\left.U_{n}\right\Vert V_{n}\right)  \leq\left(C-1\right)\sum_{j=1}^{n-1}\frac{n-j}{\left(N-j\right)^{2}}.\\
  =\left(C-1\right)\sum_{j=1}^{n-1}\left(\frac{1}{N-j}-\frac{N-n}{\left(N-j\right)^{2}}\right)\\
  =\left(C-1\right)\left(\sum_{j=1}^{n-1}\frac{1}{N-j}-\left(N-n\right)\sum_{j=1}^{n-1}\frac{1}{\left(N-j\right)^{2}}\right)\,.
\end{dmath*}
Now each of these terms can be bounded by an integral.
\begin{dmath*}
\left(C-1\right)\left(\sum_{j=1}^{n-1}\frac{1}{N-j}-\left(N-n\right)\sum_{j=1}^{n-1}\frac{1}{\left(N-j\right)^{2}}\right)\\
\leq\left(C-1\right)\left(\int_{1}^{n}\frac{1}{N-x}\,\mathrm{d}x-\left(N-n\right)\int_{0}^{n-1}\frac{1}{\left(N-x\right)^{2}}\,\mathrm{d}x\right)\\
=\left(C-1\right)\left(\ln\left(\frac{N-1}{N-n}\right)-\left(N-n\right)\left(\frac{1}{N-n+1}-\frac{1}{N}\right)\right)\\
=\left(C-1\right)\left(\ln\left(\frac{N-1}{N-n}\right)-\left(\frac{N-n+1-1}{N-n+1}-\frac{N-n}{N}\right)\right)\\
=\left(C-1\right)\left(\ln\left(\frac{N-1}{N-n}\right)+\frac{1}{N-n+1}-\frac{n}{N}\right).
\end{dmath*}
\end{proof}

\section{Asymptotic results}

The upper bounds are approximately achieved in the extreme case where
$K=1$ and $n=2.$ In this case the hypergeometric distribution is
given by $\Pr\left(U_{2}=0\right)=1-\nicefrac{2}{N}$ and $\Pr\left(U_{2}=1\right)=\nicefrac{2}{N}.$
The corresponding binomial distribution is given by $\Pr\left(V_{2}=0\right)=\left(1-\nicefrac{1}{N}\right)^{2}$
and $\Pr\left(V_{2}=1\right)=2\cdot\nicefrac{1}{N}\cdot\left(1-\nicefrac{1}{N}\right)$.
Therefore the divergence is
\begin{dmath*}
D\left(U_{2}\left\Vert V_{2}\right.\right)  =\left(1-\frac{2}{N}\right)\ln\frac{1-\frac{2}{N}}{\left(1-\nicefrac{1}{N}\right)^{2}}+\frac{2}{N}\ln\frac{\nicefrac{2}{N}}{\nicefrac{2}{N}\cdot\left(1-\nicefrac{1}{N}\right)}\\
  =-\left(1-\frac{2}{N}\right)\ln\left(1+\frac{1}{N^{2}\left(1-\frac{2}{N}\right)}\right)-\frac{2}{N}\ln\left(1-\nicefrac{1}{N}\right).
\end{dmath*}
Therefore 
\[
N^{2}\cdot D\left(U_{2}\left\Vert V_{2}\right.\right)\to-1+2=1\textrm{ for }N\to\infty.
\]
The lower bound is 
\[
D\left(U_{2}\left\Vert V_{2}\right.\right)\geq\frac{1}{2\left(N-1\right)^{2}}.
\]
Therefore we cannot have a distribution independent upper bound that
is less than the twice the lower bound. 

The lower bounds (\ref{eq:*}) and (\ref{eq:optLower}) are tight in the sense that it
has the correct asymptotic behavior if $N$ tends to infinity and
$\nicefrac{n}{N}$ converges. In order to prove this we have used
the upper bound with four terms (\ref{eq:oevre}). We will also use
a slightly different expansion. 
\begin{theorem}
\label{thm:Asymptotisk}Assume that $n_{\ell}$ and $N_{\ell}$ are
increasing sequences of natural numbers such that $n_{\ell}<N_{\ell}$
and the number of colors $C$ is fixed. Assume further that there
exists $\epsilon>0$ such that $p_{c}\geq\epsilon$ for all $\ell.$
Assume finally that $1-\frac{n_{\ell}}{N_{\ell}}\to r$ for $\ell\to\infty.$
Then
\[
D\left(U_{n_{\ell}}\left\Vert V_{n_{\ell}}\right.\right)\to\left(C-1\right)\frac{r-1-\ln\left(r\right)}{2}
\]
for $\ell\to\infty.$
\end{theorem}

\begin{proof}
First we note that 
\begin{dmath*}
D\left(\left.U_{n}\right\Vert V_{n}\right)  =D\left(\left.X_{n}\right\Vert Y_{n}\right)\\
  =\sum_{m=1}^{n-1}D\left(\left.X_{m+1}\right\Vert Y_{m+1}\left|X_{m}\right.\right)
\end{dmath*}
where 
\[
D\left(\left.X_{m+1}\right\Vert Y_{m+1}\left|X_{m}=x\right.\right)=\sum_{c=1}^{C}R\left(m,c\right)
\]
and where 
\begin{dmath*}
R\left(m,c\right)  =\sum_{h}\Pr\left(U\left(m,c\right)=h\right)p_{c}\phi\left(\frac{\frac{k_{c}-h}{N-m}}{p_{c}}\right)\,.
\end{dmath*}
First we note that
\begin{dmath*}
\frac{\frac{k_{c}-h}{N-m}}{p_{c}}  =\frac{\frac{k_{c}-h}{N-m}-p_{c}}{p_{c}}=\frac{\frac{k_{c}-h-Np_{c}+mp_{c}}{N-m}-p_{c}}{p_{c}}=\frac{mp_{c}-h}{p_{c}\left(N-m\right)}.
\end{dmath*}
Therefore 
\begin{dmath*}
R\left(m,c\right)  =\sum_{h}\Pr\left(U\left(m,c\right)\hiderel=h\right)p_{c}\phi\left(\frac{mp_{c}-h}{p_{c}\left(N-m\right)}\right)\\
 \leq\sum_{h}\Pr\left(U\left(m,c\right)\hiderel=h\right)p_{c}\left(\begin{array}{c}
\frac{1}{2}\left(\frac{mp_{c}-h}{p_{c}\left(N-m\right)}\right)^{2}\\
-\frac{1}{6}\left(\frac{mp_{c}-h}{p_{c}\left(N-m\right)}\right)^{3}\\
+\frac{1}{3}\left(\frac{mp_{c}-h}{p_{c}\left(N-m\right)}\right)^{4}
\end{array}\right)\\
  =\frac{1}{2p_{c}\left(N-m\right)^{2}}\sum_{h}\Pr\left(U\left(m,c\right)\hiderel=h\right)\left(h-mp_{c}\right)^{2}\\
  +\frac{1}{6p_{c}^{2}\left(N-m\right)^{3}}\sum_{h}\Pr\left(U\left(m,c\right)\hiderel=h\right)\left(h-mp_{c}\right)^{3}\\
  +\frac{1}{3p_{c}^{3}\left(N-m\right)^{4}}\sum_{h}\Pr\left(U\left(m,c\right)\hiderel=h\right)\left(h-mp_{c}\right)^{4}.
\end{dmath*}
These three terms are evaluated separately. 

The second order term is 
\begin{dmath*}
\frac{1}{2p_{c}\left(N-m\right)^{2}}\sum_{h}\Pr\left(U\left(m,c\right)=h\right)\left(h-mp_{c}\right)^{2}
=\frac{1}{2p_{c}\left(N-m\right)^{2}}mp_{c}\left(1-p_{c}\right)\frac{N-m}{N-1}
=\frac{m\left(1-p_{c}\right)}{2\left(N-m\right)\left(N-1\right)}.
\end{dmath*}
Summation over colors $c$ gives 
\[
\frac{m\left(C-1\right)}{2\left(N-m\right)\left(N-1\right)}.
\]
Summation over $m$ gives 
\[
\frac{C-1}{2\left(N-1\right)}\sum_{m=1}^{n-1}\frac{m}{N-m}.
\]
As $N$ tends to infinity the sum can be approximated by the integral
\begin{dmath*}
\int_{0}^{n}\frac{x}{N-x}\,\mathrm{d}x  =\left[-N\ln\left(N-x\right)-x\right]_{0}^{n}\\
  =N\ln\left(\frac{N}{N-n}\right)-n.
\end{dmath*}
Therefore 
\begin{dmath*}
\lim_{\ell\to\infty}\frac{C-1}{2\left(N-1\right)}\sum_{m=1}^{n-1}\frac{m}{N-m}  =\frac{C-1}{2}\lim_{\ell\to\infty}\frac{N\ln\left(\frac{N}{N-n}\right)-n}{N-1}\\
  =\left(C-1\right)\frac{r-1-\ln\left(r\right)}{2}.
\end{dmath*}

The third order term is 
\begin{dmath*}
\frac{1}{6p_{c}^{2}\left(N-m\right)^{3}}\sum_{h}\Pr\left(U\left(m,c\right)\hiderel=h\right)\left(h-mp_{c}\right)^{3}\\
=\frac{mp_{c}\left(1-p_{c}\right)\left(1-2p_{c}\right)\frac{\left(N-m\right)\left(N-2m\right)}{\left(N-1\right)\left(N-2\right)}}{6p_{c}^{2}\left(N-m\right)^{3}}\\
=\frac{\left(1-p_{c}\right)\left(1-2p_{c}\right)}{6p_{c}}\cdot\frac{m\left(N-2m\right)}{\left(N-m\right)^{2}\left(N-1\right)\left(N-2\right)}.
\end{dmath*}
Since $p_{c}\geq\epsilon$ we have 
\[
\sum_{c=1}^{C}\left|\frac{\left(1-p_{c}\right)\left(1-2p_{c}\right)}{6p_{c}}\right|\leq\sum_{c=1}^{C}\frac{1}{6\epsilon}=\frac{C}{6\epsilon}.
\]
Since $m\leq n$ we have 
\begin{dmath*}
\sum_{m=1}^{n-1}\left|\frac{m\left(N-2m\right)}{\left(N-m\right)^{2}\left(N-1\right)\left(N-2\right)}\right|  \leq\sum_{m=1}^{n-1}\frac{nN}{\left(N-n\right)^{2}\left(N-1\right)\left(N-2\right)}\\
  \leq\frac{n^{2}N}{\left(N-n\right)^{2}\left(N-1\right)\left(N-2\right)}.
\end{dmath*}
We see that the thrid order term tends to zero as $\ell$ tends to
$\infty.$

The fourth term is 
\[
\frac{1}{3p_{c}^{3}\left(N-m\right)^{4}}\sum_{h}\Pr\left(U\left(m,c\right)=h\right)\left(h-mp_{c}\right)^{4}.
\]
Using the formula for the fourth central moment of the hypergeometric
distribution we get
\begin{dmath*}
\left(\frac{\begin{array}{c}
\left(N-1\right)\left(\begin{array}{c}
N\left(N-1\right)-6m\left(N-m\right)\\
6N^{2}p_{c}\left(1-p_{c}\right)
\end{array}\right)\\
6mp_{c}\left(1-p_{c}\right)\left(N-m\right)\left(5N-6\right)
\end{array}}{mp_{c}\left(1-p_{c}\right)\left(N-m\right)\left(N-2\right)\left(N-3\right)}+3\right)\frac{\left(mp_{c}\left(1-p_{c}\right)\frac{N-m}{N-1}\right)^{2}}{3p_{c}^{3}\left(N-m\right)^{4}}\leq
\left(\frac{N^{3}+30np_{c}\left(1-p_{c}\right)N^{2}}{mp_{c}\left(1-p_{c}\right)\left(N-n\right)\left(N-2\right)\left(N-3\right)}+3\right)\frac{m^{2}p_{c}^{2}\left(1-p_{c}\right)^{2}}{3p_{c}^{3}\left(N-n\right)^{4}}\leq
\frac{\left(N^{3}+30np_{c}\left(1-p_{c}\right)N^{2}\right)m}{3p_{c}^{2}\left(N-2\right)\left(N-3\right)\left(N-n\right)^{5}}+\frac{m^{2}}{p_{c}\left(N-n\right)^{4}}\leq
\frac{nN^{3}+30n^{2}N^{2}}{\epsilon^{2}\left(N-2\right)\left(N-3\right)\left(N-n\right)^{5}}+\frac{n^{2}}{\epsilon\left(N-n\right)^{4}}.
\end{dmath*}
Summation over $c$ and $n$ has the effect of multiplying by $\left(C-1\right)\left(n-1\right).$
Since the numerators are of lower degree than the denominators the
fourth order term will tend to zero as $\ell$ tend to $\infty.$
\end{proof}

\section*{Acknowledgement }

A draft of this manuscript was ready before the tragic death of my
dear friend and colleague Franti{\v s}ek Mat{\' u}{\v s} (Fero). He was a perfectionist
and was not satisfied with certain technical details and with the notation.
I hope the manuscript in its present form will live up to his high
standards.

\appendix

\section{Bounding Taylor polynomials \label{sec:BoundingTaylorpolynomials}}

Let $\phi\left(x\right)=x\ln x-\left(x-1\right)$ with the convention
that $\phi\left(0\right)=1.$ Then the derivatives are

\begin{dgroup*}
\begin{dmath*}\phi\left(x\right)  =x\ln\left(x\right)-\left(x-1\right)\,,\end{dmath*}
\begin{dmath*}\phi'\left(x\right)  =\ln\left(x\right)\,,\end{dmath*}
\begin{dmath*}\phi''\left(x\right)  =\frac{1}{x}\,,\end{dmath*}
\begin{dmath*}\phi^{\left(3\right)}\left(x\right)  =\frac{\textrm{-}1}{x^{2}}\,,\end{dmath*}
\begin{dmath*}\phi^{\left(4\right)}\left(x\right)  =\frac{2}{x^{3}}\,,\end{dmath*}
\begin{dmath*}\phi^{\left(5\right)}\left(x\right)  =\frac{\textrm{-}6}{x^{4}}\,.\end{dmath*}
\end{dgroup*}
Evaluations at $x=1$ give
\begin{dgroup*}
\begin{dmath*}\phi\left(1\right)  =0\,,\end{dmath*}
\begin{dmath*}\phi'\left(1\right)  =0\,,\end{dmath*}
\begin{dmath*}\phi''\left(1\right)  =1\,,\end{dmath*}
\begin{dmath*}\phi^{\left(3\right)}\left(1\right)  =\textrm{-}1\,,\end{dmath*}
\begin{dmath*}\phi^{\left(4\right)}\left(1\right)  =2\,,\end{dmath*}
\begin{dmath*}\phi^{\left(5\right)}\left(1\right)  =\textrm{-}6\,.\end{dmath*}
\end{dgroup*}
Since the even derivates are positive the odd Taylor polynomials give
lower bounds, so we have 
\begin{dgroup*}
\begin{dmath*}
\phi\left(x\right)  \geq0\,,\end{dmath*}
\begin{dmath*}\phi\left(x\right)  \geq\frac{1}{2}\left(x-1\right)^{2}-\frac{1}{6}\left(x-1\right)^{3}.
\end{dmath*}
\end{dgroup*}
Since the odd derivates are negative we have 
\begin{dgroup*}
\begin{dmath*}\phi\left(x\right)  \leq\frac{1}{2}\left(x-1\right)^{2},\end{dmath*}
\begin{dmath*}\phi\left(x\right)  \leq\frac{1}{2}\left(x-1\right)^{2}-\frac{1}{6}\left(x-1\right)^{3}+\frac{1}{12}\left(x-1\right)^{4},\end{dmath*}
\end{dgroup*}
for $x\geq1$ and the reversed inequalities for $x\leq1$. We will
add a positive term to these inequalities in order to get an upper
bound that holds for all $x\geq0.$ The inequalities are
\begin{dgroup*}
\begin{dmath*}\phi\left(x\right)  \leq\left(x-1\right)^{2},\end{dmath*}
\begin{dmath*}\phi\left(x\right)  \leq\frac{1}{2}\left(x-1\right)^{2}-\frac{1}{6}\left(x-1\right)^{3}+\frac{1}{3}\left(x-1\right)^{4}.
\end{dmath*}
\end{dgroup*}
 We have to prove these inequalities in the interval $[0,1]$.

For the first inequality we define $g\left(x\right)=\left(x-1\right)^{2}-f\left(x\right),$
and have to prove that this function is non-negative. We have $g(0)=g(1)=0$
so it is sufficient to prove that $g$ is first increasing and then
decreasing, or equivalently that $g'$ is first positive and then
negative. We have $g'\left(x\right)=2\left(x-1\right)-\ln x$ so that
$g'\left(x\right)\to\infty$ for $x\to0$ and $g'\left(1\right)=0.$
Therefore it is sufficient to prove that $g'$ is first decreasing
and then increasing. $g''\left(x\right)=2-\nicefrac{1}{x},$ which
is negative for $x<\nicefrac{1}{2}$ and positive for $x>\nicefrac{1}{2}$.
The second inequality is proved in the same way except that we have
to differentiate four times. 

\section{Proof of Lemma 4.2\label{sec:Proof-of-Lemma}}

For $j=1$ we have
\begin{dmath*}
I\left(X^{j},X_{j+1}\mid X^{j-1}\right)  =I\left(X_{1},X_{2}\right)\\
  =D\left(\left.\left.X_{2}\right\Vert Y_{2}\right|X_{1}\right).
\end{dmath*}
Now 
\[
D\left(\left.\left.X_{2}\right\Vert Y_{2}\right|X_{1}\right)=\sum_{x_{1}}\Pr\left(X_{1}=x_{1}\right)D\left(\left.\left.X_{2}\right\Vert Y_{2}\right|X_{1}=x_{1}\right)
\]
and 
\[
D\left(\left.\left.X_{2}\right\Vert Y_{2}\right|X_{1}=x_{1}\right)=\sum_{c=1}^{C}R\left(1,c\right)
\]
where
\[
R\left(1,c\right)=\sum_{h}\Pr\left(U\left(1,c\right)=h\right)p_{c}\phi\left(\frac{\frac{k_{c}-h}{N-1}}{p_{c}}\right).
\]
Using the upper bound (\ref{eq:chi}) we get
\begin{dmath*}
R\left(1,c\right)  \leq\sum_{h}\Pr\left(U\left(1,c\right)\hiderel=h\right)p_{c}\left(\frac{\frac{k_{c}-h}{N-1}}{p_{c}}-1\right)^{2}
  =\frac{p_{c}\left(1-p_{c}\right)}{p_{c}\left(N-1\right)^{2}}\\
  =\frac{1-p_{c}}{\left(N-1\right)^{2}}.
\end{dmath*}
Summation over the colors gives 
\[
D\left(\left.\left.X_{m+1}\right\Vert Y_{m+1}\right|X^{m}\right)\leq\frac{C-1}{\left(N-1\right)^{2}}\,.
\]

\begin{proof}
In order to get a lower bound we calculate 
\begin{dmath*}
\sum_{h}\Pr\left(U\left(1,c\right)\hiderel=h\right)p_{c}\left(\frac{1}{2}\left(\frac{\frac{k_{c}-h}{N-1}}{p_{c}}-1\right)^{2}-\frac{1}{6}\left(\frac{\frac{k_{c}-h}{N-1}}{p_{c}}-1\right)^{3}\right)
=\sum_{h}\Pr\left(U\left(1,c\right)\hiderel=h\right)\frac{\left(h-p_{c}\right)^{2}}{2p_{c}\left(N-1\right)^{2}}+\sum_{h}\Pr\left(U\left(1,c\right)\hiderel=h\right)\frac{\left(h-p_{c}\right)^{3}}{6p_{c}^{2}\left(N-1\right)^{3}}.
\end{dmath*}
These terms will be evaluated separately. 

As before 
\[
\sum_{h}\Pr\left(U\left(1,c\right)=h\right)\frac{\left(h-p_{c}\right)^{2}}{2p_{c}\left(N-1\right)^{2}}=\frac{\left(n-j\right)\left(1-p_{c}\right)}{2\left(N-1\right)^{2}}.
\]
Summation over $c$ gives 
\[
\left(C-1\right)\sum_{j=1}^{n-1}\frac{n-j}{\left(N-j\right)^{2}}.
\]

The third term is 
\begin{dmath*}
\sum_{h}\Pr\left(U\left(1,c\right)=h\right)\frac{\left(h-p_{c}\right)^{3}}{6p_{c}^{2}\left(N-1\right)^{3}}  =\frac{p_{c}\left(1-p_{c}\right)\left(1-2p_{c}\right)\frac{\left(N-1\right)\left(N-2\right)}{\left(N-1\right)\left(N-2\right)}}{6p_{c}^{2}\left(N-1\right)^{3}}\\
  =\frac{\left(1-p_{c}\right)\left(1-2p_{c}\right)}{6p_{c}\left(N-1\right)^{3}}.
\end{dmath*}
Summation over $c$ gives
\[
\frac{Q}{6\left(N-1\right)^{3}}
\]
where 
\begin{dmath*}
Q =\sum_{c=1}^{C}\frac{\left(1-p_{c}\right)\left(1-2p_{c}\right)}{p_{c}}
  =\sum_{c=1}^{C}\left(\frac{1}{p_{c}}-3+2p_{c}\right)
  =\sum_{c=1}^{C}\frac{1}{p_{c}}-3C+2\,.
\end{dmath*}
We introduce
\begin{dmath*}
\chi^{2}\left(\frac{1}{C},p_{c}\right)  =\sum_{c=1}^{C}\frac{\left(\frac{1}{C}-p_{c}\right)^{2}}{p_{c}}
  =\sum_{c=1}^{C}\left(\frac{1}{C^{2}}\cdot\frac{1}{p_{c}}-\frac{2}{C}+p_{c}\right)
 =\frac{1}{C^{2}}\sum_{c=1}^{C}\frac{1}{p_{c}}-1
\end{dmath*}

so that
\begin{dmath*}
Q  =C^{2}\chi^{2}\left(\frac{1}{C},p_{c}\right)+C^{2}-3C+2
  =C^{2}\chi^{2}\left(\frac{1}{C},p_{c}\right)+\left(C-1\right)\left(C-2\right)
  \geq0.
\end{dmath*}
\end{proof}


\begin{thebibliography}{1}

\bibitem{Harremoes2014}
Harremo{\"e}s, P.,
\newblock Mutual information of contingency tables and related inequalities,
\newblock in {\em 2014 IEEE International Symposium on Information Theory},
  pages 2474--2478, IEEE, 2014.

\bibitem{Barbour1992}
Barbour, A.~D., Holst, L., and Janson, S.,
\newblock {\em Poisson Approximation},
\newblock Oxford Studies in Probability 2, Clarendon Press, Oxford, 1992.

\bibitem{Stam1978}
Stam, A.~J.,
\newblock Statistica Neerlandica {\bf 32} (1978) 81.

\bibitem{Cover1991}
Cover, T.~M. and Thomas, J.~A.,
\newblock {\em Elements of Information Theory},
\newblock Wiley, 1991.

\bibitem{Csiszar2004}
Csisz{\'a}r, I. and Shields, P.,
\newblock {\em Information Theory and Statistics: A Tutorial},
\newblock Foundations and Trends in Communications and Information Theory, Now
  Publishers Inc., 2004.

\bibitem{Matus2017}
Mat{\' u}{\v s}, F.,
\newblock Urns and entropies revisited,
\newblock in {\em 2017 IEEE International Symposium on Information Theory
  (ISIT),}, pages 1451--1454, 2017.

\bibitem{Harremoes2004}
Harremo{\"e}s, P. and Ruzankin, P.,
\newblock IEEE Trans. Inform Theory {\bf 50} (2004) 2145.

\bibitem{Diaconis1987}
Diaconis, P. and Friedman, D.,
\newblock Ann. Inst. Henri Poincar{\'e} {\bf 23-2} (1987) 397.

\bibitem{Harremoes2016b}
Harremo{\"e}s, P., Johnson, O., and Kontoyiannis, I.,
\newblock Thinning and information projections,
\newblock arXiv:1601.04255, 2016.

\end{thebibliography}


\end{document}